\documentclass[preprint,11pt]{elsarticle}

\usepackage{amssymb}
\usepackage{amsfonts}
\usepackage{amsmath}
\usepackage{float}
\usepackage{color}
\usepackage{url}
\usepackage{epstopdf}
\newtheorem{lemma}{Lemma}

\newtheorem{theorem}{Theorem}
\newtheorem{proposition}{Proposition}
\newtheorem{corollary}{Corollary}

\newcommand{\Z}{ \mathbb{Z} }
\usepackage{setspace}
\usepackage{tcolorbox}
\usepackage{epstopdf}
\usepackage{subcaption}
\journal{Statistics and Probability Letters}

\begin{document}

\begin{frontmatter}



\title{A note on strong-consistency of   componentwise ARH(1) predictors}


\author{M. D. Ruiz-Medina  and J. \'Alvarez-Li\'ebana}

\address{Department of Statistics and Operation
Research (mruiz@ugr.es,  javialvaliebana@ugr.es)\\  Faculty of Sciences,  University of Granada\\
Campus Fuente Nueva s/n\\18071 Granada, Spain}


\begin{abstract}
New results on strong-consistency in the  trace operator norm are obtained, in the parameter estimation of  an autoregressive Hilbertian process of order one (ARH(1) process). Additionally, a strongly-consistent  diagonal componentwise estimator of  the autocorrelation operator is derived, based on its empirical singular value decomposition.
\end{abstract}

\begin{keyword}
Dimension reduction techniques  \sep empirical orthogonal bases \sep functional prediction \sep   
 strong-consistency \sep trace norm

\vspace*{0.5cm}



\noindent \textit{2010 Mathematics Subject Classification}
Primary 60G10; 60G15. Secondary 60F99; 60J05; 65F15
\end{keyword}

\end{frontmatter}


\section{Introduction.}
\label{sec:1}
 
There exists an extensive literature on Functional Data Analysis (FDA) techniques. In the past few years, the primary focus
of FDA was mainly on independent and identically distributed  (i.i.d.)  functional observations.  
The classical  book by Ramsay and Silverman \cite{RamsaySilverman05}  provides a wide overview on FDA techniques (e.g., regression, principal components analysis, linear modeling,  canonical correlation analysis, curve registration, and principal differential analysis, etc). An introduction to  nonparametric statistical approaches for FDA
can be found in Ferraty and  Vieu \cite{Ferraty06}. 
 We also refer to the recent  monograph by  Hsing and Eubank \cite{HsingEubank15}, where the usual functional analytical tools  in FDA are introduced,  addressing several statistical and estimation problems for  random elements in function spaces.  
Special attention is paid to the monograph  by Horv\'ath and  Kokoszka \cite{HorvathandKokoszka} covering functional inference based on second order statistics. 

We  refer  the reader to the methodological  survey paper by
Cuevas  \cite{Cuevas14},  covering nonparametric techniques and  discussing central topics in FDA. Recent advances on statistics in high/infinite  dimensional spaces are collected in the IWFOS'14 Special Issue published in the Journal of Multivariate Analysis (see  Goia and Vieu  \cite{GoiaVieu16} who summarized its contributions, providing a brief discussion on the current literature).

 A central issue in FDA is to take into account the temporal dependence of the observations. Although the literature on scalar and vector time series is huge,
there are relatively few contributions dealing with functional time series, and, in general, with dependent functional data.  For instance, Part III (Chapters 13--18) of the monograph by  
Horv\'ath and  Kokoszka \cite{HorvathandKokoszka} is devoted to
this issue, including  topics related to functional  time series (in particular, the functional autoregressive model), and the statistical analysis of spatially distributed functional data.
The moment-based notion of weak dependence introduced in 
H\"ormann and Kokoszka \cite{HormannKokoszka10} is also  accommodated to the statistical analysis of functional time series. This notion does not require the specification of a data model, and  
can be used to study the properties of many nonlinear sequences (see e.g.,
H\"ormann \cite{H2008}; Berkes et al. \cite{Berkes}, for recent applications).

This paper adopts the methodological approach presented in
Bosq \cite{Bosq00} for functional time series.
 That monograph  studies the theory of linear functional time series,
both in Hilbert and Banach spaces, focusing on the functional autoregressive model.  Several authors have studied the asymptotic properties of componentwise estimators of the autocorrelation operator of an ARH(1) process, and of the associated plug-in predictors.   We refer to \cite{Guillas01,Mas99,Mas04,Mas07}, where the efficiency, consistency and asymptotic normality of these estimators are addressed, in a parametric framework (see also 
\'Alvarez-Li\'ebana, Bosq and  Ruiz-Medina  \cite{Alvarez1600}, on estimation of the Ornstein-Uhlenbeck processes in  Banach spaces, and    \cite{Alvarez16}, on weak consistency in the   Hilbert-Schmidt operator norm of componentwise estimators). Particularly, strong-consistency in the  norm of the space of bounded linear operators was derived in \cite{Bosq00}. In the derivation of these results, the autocorrelation operator is usually assumed to be a Hilbert-Schmidt operator,  when the eigenvectors of the autocovariance operator are unknown.   This paper proves that, under basically the same setting of conditions as in the cited papers, the  componentwise estimator  of the autocorrelation operator proposed in  \cite{Bosq00}, based on  the   empirical eigenvectors of the autocovariance operator, is also strongly-consistent in the  
 Hilbert-Schmidt and trace operator norms.

 The dimension reduction problem  constitutes also a central topic in the parametric, nonparametric and semiparametric FDA statistical frameworks. Special attention to this topic has been paid, for instance,  in the context of functional regression with functional response and functional predictors  (see, for example, Ferraty et al. \cite{FerratyKeilegom12}, where asymptotic
normality is derived, and, Ferraty et al. \cite{Ferraty02}, in the functional time series framework). In the semiparametric and nonparametric estimation techniques, a kernel-based formulation is usually adopted. Real-valued covariates were incorporated in the novel semiparametric kernel-based proposal by Aneiros-P\'erez and  Vieu \cite{Aneiros08}, providing  an extension to the functional partial linear time series framework (see also   Aneiros-P\'erez and  Vieu \cite{AneirosVieu06}). 
Motivated by spectrometry applications, a two-terms Partitioned Functional Single Index Model is introduced in Goia and  Vieu
\cite{GoiaVieu15}, in a semiparametric framework. In the ARH(1) process framework, the present  paper  provides 
 a new diagonal componentwise estimator of the autocorrelation operator, based on its empirical  singular value decomposition. Its strong-consistency is proved as well. The diagonal  design leads to an important dimension reduction, going beyond the usual isotropic restriction on the kernels involved in the approximation of the regression operator (respectively, autocorrelation operator), in the nonparametric framework.  Recently,  Petrovich and Reimherr  \cite{PetrovichReimherr17} address the dimension reduction provided by the functional principal component projections in the general case when eigenvalues can be repeated, instead of the classical assumptions that their multiplicity should be one.
    
The outline of the paper is the following. Section \ref{sec:2} introduces basic definitions and preliminary results. Section \ref{sec:4} derives strong-consistency of the estimator introduced in Bosq \cite{Bosq00}, in the  trace  norm. Section \ref{dce} formulates a 
strongly-consistent diagonal componentwise estimator of the autocorrelation operator. Proofs of the results  are given in Section \ref{secproof}.
     
\section{Preliminaries.}
\label{sec:2}
 Let  $H$ be a real separable Hilbert space, and let  $X = \left\lbrace X_n,\ n \in \mathbb{Z} \right\rbrace $
 be  a zero-mean ARH(1) process on the probability space $(\Omega,\mathcal{A},P),$ satisfying:
\begin{equation}
X_{n } = \rho \left(X_{n-1} \right) + \varepsilon_{n},\ n \in
\Z,\label{24bb}
\end{equation}
\noindent  where $\rho \in \mathcal{L}(H),$ with $\mathcal{L}(H)$ being the space of bounded linear operators, with the uniform norm $\left\| \mathcal{A} \right\|_{\mathcal{L}(H)}=\sup_{f\in H; \ \|f\|_{H}\leq 1}\mathcal{A}(f),$
 for every $\mathcal{A}\in \mathcal{L}(H).$
 In our case, $\rho \in \mathcal{L}(H)$ satisfies $\Vert \rho^{k}
\Vert_{\mathcal{L}\left(H\right)}<1,$  for $k\geq k_{0},$ and for some
 $k_{0},$ where $\rho^{k}$ denotes the $k$th power of $\rho,$ i.e., the composition operator $\rho\underset{k}{\dots}\rho.$  The $H$-valued
innovation process $\varepsilon= \left \lbrace \varepsilon_{n}, \ n\in \mathbb{Z}\right\rbrace$ is assumed to be a strong white noise, and to be  uncorrelated with the random
initial condition.  $X$ then  admits the MAH($\infty$) representation  $X_n = \displaystyle \sum_{k=0}^{\infty} \rho^{k} \left( \varepsilon_{n-k} \right),$ for $n \in \mathbb{Z},$ 
  providing   the unique stationary solution to equation (\ref{24bb}) (see
 \cite{Bosq00}).

The trace autocovariance operator  of  $X$ is given by $C_{X}={\rm E}[X_{n}\otimes X_{n}]=
{\rm E}[X_{0}\otimes X_{0}],$ for $n\in \mathbb{Z},$
 and its empirical version $\mathcal{C}_{n}$ is defined as
\begin{equation}
\mathcal{C}_{n} = \frac{1}{n} \displaystyle \sum_{i=0}^{n-1} X_i \otimes X_i,\quad n \geq 2,  \label{66aa}
\end{equation}
\noindent  where, for  $f\in H,$ and  $i,j\in \mathbb{N},$ the random operator $X_{i} \otimes X_{j}$ is given by  $\left(X_{i} \otimes X_{j} \right) (f)= \langle X_{i},f \rangle_H X_{j}.$ 
 In the following,   $\left\lbrace C_j, \ j \geq 1 \right\rbrace$ and $\{\phi_{j},\ j\geq 1\}$  denote  the respective sequence of eigenvalues and eigenvectors  of the autocovariance operator $C_{X},$ satisfying
$C_{X}(\phi_{j})=C_{j}\phi_{j},$ for  $j\geq 1.$ 
Also, by $\{C_{n,j}, \ j \geq 1\}$ and $\{\phi_{n,j},\ j\geq 1\}$ we respectively   denote the empirical eigenvalues and  eigenvectors  of $\mathcal{C}_{n}$    (see \cite{Bosq00}, pp. 102--103), 
 \begin{eqnarray}
&&\hspace*{-1.2cm}\mathcal{C}_{n}\phi_{n,j}= C_{n,j}\phi_{n,j},\ j\geq 1,\quad C_{n,1}\geq \dots\geq C_{n,n}\geq 0=C_{n,n+1}=C_{n,n+2}\dots  \label{ev1}
\end{eqnarray}
\noindent 
Consider now   the nuclear  cross-covariance operator 
$D_{X}={\rm E}[X_{i}\otimes X_{i+1}]={\rm E}[X_{0}\otimes X_{1}],$  $i\in\mathbb{Z},$  and   its empirical version $\mathcal{D}_n = \frac{1}{n-1}\displaystyle \sum_{i=0}^{n-2} X_i \otimes X_{i+1},$   $n \geq 2.$ 

The following assumption will appear in the subsequent development.

\medskip

\noindent  \textbf{Assumption A1.} The   random initial condition  $X_{0}$ of  $X$ in (\ref{24bb}) satisfies
$ \left\| X_0 \right\|_{H}  < M,\quad a.s.,$  for some $M.$
Here, a.s. denotes almost surely. 

\begin{theorem}
\label{theorem2} (see Theorem 4.1 on pp. 98--99, Corollary 4.1 on pp. 100--101 and  Theorem 4.8 on pp. 116--117, in \cite{Bosq00}).
 If ${\rm E} \left[ \left\| X_0 \right\|_{H}^{4} \right] < \infty,$ for any $\beta > \frac{1}{2},$ as $n\rightarrow \infty,$
\begin{equation}
\frac{n^{1/4}}{\left(\ln(n) \right)^{\beta}}  \left\| \mathcal{C}_{n} - C_{X} \right\|_{\mathcal{S}(H)} \to^{a.s.}0, \quad \frac{n^{1/4}}{\left(\ln(n) \right)^{\beta}} \left\| \mathcal{D}_n - D_{X} \right\|_{\mathcal{S}(H)} \to^{a.s.}0,
\end{equation}
\noindent where  $\to^{a.s.}$ means almost surely convergence.
Under \textbf{Assumption A1}, 
\begin{eqnarray}
& &\left\| \mathcal{C}_{n} - C_{X} \right\|_{\mathcal{S}(H)} = \mathcal{O} \left(\left(\frac{\ln(n) }{n} \right)^{1/2} \right)~a.s.,
\nonumber\\
& &\left\| \mathcal{D}_n - D_{X} \right\|_{\mathcal{S}(H)} = \mathcal{O} \left(\left(\frac{\ln(n) }{n} \right)^{1/2} \right)~a.s.,
\end{eqnarray}
\noindent where $\left\| \cdot \right\|_{\mathcal{S}(H)}$ is the Hilbert-Schmidt operator norm. \end{theorem}
 Let $k_n$ be a truncation parameter such that 
$\lim_{n \to \infty} k_n = \infty,$ $\frac{k_n}{n} < 1,$ and 
\begin{equation}
\Lambda_{k_n}=\sup_{1\leq j\leq k_n}(C_{j}-C_{j+1})^{-1}.\label{uee}
 \end{equation}
\section{Strong-consistency  in the  trace operator norm}
\label{sec:4}
This section derives the strong-consistency of the componentwise estimator $\widetilde{\rho}_{k_{n}}$ (see equation (\ref{140}) below), in the trace norm, which also implies its strong-consistency in the Hilbert-Schmidt operator norm.  As it is well-known,  
 for a trace operator $\mathcal{K}$ on $H,$ its trace  norm $\|\mathcal{K}\|_{1}$  is finite,  and, for an orthonormal basis $\{\varphi_{n},\ n\geq 1\}$ of $H,$  such a norm  is  given by 
\begin{equation}
\|\mathcal{K}\|_{1}=\sum_{n=1}^{\infty }\left\langle \sqrt{\mathcal{K}^{\star}\mathcal{K}}(\varphi_{n} ),\varphi_{n}\right\rangle_{H}.\label{tn}
\end{equation}

 In  Theorem \ref{proposition5} below,  the following lemma will be applied:
\begin{lemma}
\label{lemeera} Under \textbf{Assumption A1},  if, as $n\to \infty,$ $k_{n}\Lambda_{k_n}=o\left(\sqrt{\frac{n}{\ln(n)}}\right),$ 
\begin{equation}\sup_{x\in H,\ \|x\|_{H}\leq 1}\left\|\rho(x)-\sum_{j=1}^{k_{n}}\left\langle \rho(x),\phi_{n,j}\right\rangle_{H}\phi_{n,j}\right\|_{H}\to_{a.s.} 0,\quad n\to \infty. \label{eqasconveee}
\end{equation}

\end{lemma}

The proof of this lemma is given in Section \ref{secproof}.

The following  condition is assumed  in the remainder of this section:

\medskip

\noindent \textbf{Assumption A2.} The empirical eigenvalue $C_{n,k_n} > 0~a.s,$ where $k_n$ denotes the truncation parameter  introduced in the previous section.

\medskip

Under \textbf{Assumption A2}, from the  observations of  $X_{0},\dots, X_{n-1},$ consider the   componentwise estimator $\widetilde{\rho}_{k_{n}}$ of $\rho$ (see (8.59)  p.218 in \cite{Bosq00})
\begin{eqnarray}
&&\widetilde{\rho}_{k_n}(x)\underset{H}{=} \widetilde{\pi}^{k_{n}}
\mathcal{D}_{n}[\mathcal{C}_{n}[\widetilde{\pi}^{k_{n}}]^{\star}]^{-1}(x)=\widetilde{\pi}^{k_{n}}
\mathcal{D}_{n}\widetilde{\mathcal{C}}_{n}^{-1}(x)\nonumber\\
&&\underset{H}{=} \displaystyle \sum_{j=1}^{k_n}\sum_{p=1}^{k_n}\left\langle \mathcal{D}_{n}\mathcal{C}_{n}^{-1}(\phi_{n,j}),\phi_{n,p}\right\rangle_{H} \phi_{n,p}\left\langle \phi_{n,j}, x\right\rangle_{H},\quad \forall x\in H,\label{140}
\end{eqnarray}
\noindent where 
$\widetilde{\mathcal{C}}_{n}^{-1}$  is  the inverse of the restriction of $\mathcal{C}_{n}$ to its principal eigenspace
of dimension $k_{n},$ which 
 is   bounded under \textbf{Assumption A2}. Here, $[\widetilde{\pi}^{k_{n}}]^{\star}$  denotes the projection operator into  $\overline{\mbox{Sp}}^{\|\cdot\|_{H}}\{\phi_{n,j};\ j=1,\dots k_{n}\}\subseteq H,$ the principal eigenspace
of dimension $k_{n},$ and  $\widetilde{\pi}^{k_{n}}$ is its adjoint or inverse.
\begin{theorem}
\label{proposition5}
Let $\rho \in \mathcal{L}(H)$ be the autocorrelation operator defined as before.
Assume  $\Lambda_{k_{n}}$
in  (\ref{uee}) satisfies  $\sqrt{k_{n}}\Lambda_{k_{n}}=o\left(\frac{n^{1/4}}{(\ln(n))^{\beta }} \right)$ as $n\rightarrow \infty ,$  for $\beta >1/2.$ 
Then, for $\widetilde{\rho}_{k_n}$ in (\ref{140}), the following assertions hold:

\noindent \textbf{(i)}  If ${\rm E} \left[ \left\| X_0 \right\|_{H}^{4} \right] < \infty,$ under \textbf{Assumption A2}, 
  \begin{equation}\|\widetilde{\rho}_{k_n}-\widetilde{\pi}^{k_{n}}\rho [\widetilde{\pi}^{k_{n}}]^{\star}\|_{1}\to^{a.s.} 0,\quad n\rightarrow \infty.\label{mreq1}
\end{equation}

\medskip

\noindent \textbf{(ii)}  Under \textbf{Assumptions A1-A2}, 
if $\rho$ is a trace operator,   then,
 \begin{equation}\|\widetilde{\rho}_{k_n}-\rho \|_{1}\to^{a.s.} 0,\quad n\to \infty.\label{mreq1c}
\end{equation}
\end{theorem}
The proof of this result is given in  Section \ref{secproof}.

The strong consistency in $H$ of the associated ARH(1) plug-in predictor $\widetilde{\rho}_{k_n}(X_{n-1})$
of $X_{n}$ then  follows (see also \cite{Bosq00} and Section \ref{secproof}).
\section{A strongly-consistent diagonal componentwise estimator}
\label{dce}
In this section, we consider the following assumption:

\medskip

\noindent \textbf{Assumption A3}. Assume  that $C_{X}$ is strictly positive, i.e.,
$C_{j}>0,$ for every $j\geq 1,$ and $D_{X}$ is a nuclear operator such that $\rho=D_{X}C_{X}^{-1}$ is  compact. 

\medskip

Under \textbf{Assumption A3},    $\rho $  admits   the  singular value decomposition(svd)  
\begin{equation}
\rho (x)\underset{H}{=}\sum_{j=1}^{\infty}\rho_{j}\left\langle x,\psi_{j}
\right\rangle_{H}\widetilde{\psi}_{j},\quad \forall  x\in H,
\label{eqrho}
\end{equation}
\noindent where, for every $j\geq 1,$ $\rho (\psi_{j})=\rho_{j} \widetilde{\psi}_{j},$ with $\rho_{j}\in \mathbb{C}$ being the singular value, and $\psi_{j}$ and 
$\widetilde{\psi}_{j}$  the right and left eigenvectors,  respectively.
Since $D_{X}$ is a nuclear operator, it  admits the  svd
\ $D_{X}(h)\underset{H}{=}\sum_{j=1}^{\infty}d_{j}\left\langle h,\varphi_{j}\right\rangle_{H}\widetilde{\varphi}_{j},$  $ h\in H,$ where $\{\varphi_{j},\ j\geq 1\}$ and $\{\widetilde{\varphi}_{j},\ j\geq 1\}$ are the respective right and left eigenvectors of $D_{X},$
 and $d_{j},$ $j\geq 1,$ are the singular values.  
  $\mathcal{D}_{n}$ is also  nuclear, and 
$\mathcal{D}_{n}(h)\underset{H}{=}\sum_{j=1}^{\infty}d_{n,j}\left\langle h,\varphi_{n,j}\right\rangle_{H}\widetilde{\varphi}_{n,j},$  $h\in H,$ with $\{\varphi_{n,j},\ j\geq 1\}$ and $\{\widetilde{\varphi}_{n,j},\ j\geq 1\}$ being 
the right and left eigenvectors, respectively,  and  $d_{n,j},$ $j\geq 1,$  the singular values.
Applying Lemma 4.2, on p. 103, in \cite{Bosq00},  
\begin{eqnarray}
&&\sup_{j\geq 1}|C_{j}-C_{n,j}|\leq \|C_{X}-\mathcal{C}_{n}\|_{\mathcal{L}(H)}\leq 
\|C_{X}-\mathcal{C}_{n}\|_{\mathcal{S}(H)}\to_{a.s.} 0,\ n\to \infty
\nonumber\\ &&\sup_{j\geq 1}|d_{j}-d_{n,j}|\leq 
\|D_{X}-\mathcal{D}_{n}\|_{\mathcal{S}(H)}\to_{a.s.} 0,\ n\to \infty.
\label{csvd}
\end{eqnarray}
From Theorem \ref{theorem2} (see equation (\ref{csvd})), under the conditions assumed in such a theroem, for $n$ sufficiently large, in view of \textbf{Assumption A3}, 
the composition operator $\mathcal{D}_{n}\mathcal{C}_{n}^{-1}$ is compact  on $H,$ admitting the svd
\begin{equation}\mathcal{D}_{n}\mathcal{C}_{n}^{-1}(h)=\sum_{j=1}^{n}\widehat{\rho}_{n,j}
\widetilde{\psi}_{n,j}\left\langle h,\psi_{n,j}\right\rangle_{H},\quad \forall h\in H,\label{svde2}
\end{equation}
\noindent where  $\mathcal{D}_{n}\mathcal{C}_{n}^{-1}(\psi_{n,j})= 
\widehat{\rho}_{n,j}\widetilde{\psi}_{n,j},$ for $j=1,\dots,n,$
 with $\{\psi_{n,j},\ j\geq 1\}$ and 
$\{\widetilde{\psi}_{n,j},\ j\geq 1\}$ being the empirical  right and left
eigenvectors of $\rho.$  
\begin{proposition}\label{remfv}
Under   conditions  in 
Theorem \ref{proposition5}\textbf{(ii)}, and \textbf{Assumption A3},
\begin{equation}\|\mathcal{D}_{n}\mathcal{C}_{n}^{-1}-D_{X}C_{X}^{-1}\|_{\mathcal{L}(H)}\to_{a.s.} 0,\quad n\to \infty.\label{scempr}
\end{equation}

\end{proposition}

The proof of this proposition directly follows from 
\begin{eqnarray}
&&\sup_{x\in H:\|x\|_{H}\leq 1}\|\mathcal{D}_{n}\mathcal{C}_{n}^{-1}(x)-D_{X}C_{X}^{-1}(x)\|_{H}\nonumber\\&&\leq 2
\|\mathcal{D}_{n}\mathcal{C}_{n}^{-1}\|_{\mathcal{L}(H)}\left[\sum_{j=1}^{k_{n}}\|\phi_{n,j}^{\prime }-\phi_{n,j}\|_{H}+\sum_{j=k_{n}+1}^{\infty}\left\|\phi_{n,j}^{\prime}\right\|_{H}\right]\nonumber\\&&
+\|\widetilde{\rho}_{k_n}-D_{X}C_{X}^{-1}\|_{\mathcal{L}(H)}\to_{a.s.} 0,\quad n\to 
\infty, \label{eremfv}
\end{eqnarray}
\noindent where $\phi_{n,j}^{\prime} =sgn\langle \phi_{j}, \phi_{n,j} \rangle_{H} \phi_{j},$ with $sgn \langle \phi_{j}, \phi_{n,j} \rangle_{H} =  \boldsymbol{1}_{\langle \phi_{j}, \phi_{n,j} \rangle_H \geq 0} -$   \linebreak $\boldsymbol{1}_{\langle \phi_{j}, \phi_{n,j} \rangle_H < 0}.$
Under \textbf{Assumption A3}, equation (\ref{scempr}) holds, if the conditions assumed in \cite{Bosq00} for the strong-consistency of $\widetilde{\rho}_{k_n}$ in $\mathcal{L}(H)$ hold.   
From Proposition  \ref{remfv}, and  
(\ref{eqrho}) and (\ref{svde2}), 
applying Lemma 4.2, on p. 103 in \cite{Bosq00},
\begin{eqnarray}
 \sup_{j\geq 1}\left|\widehat{\rho}_{n,j}-\rho_{j}\right|
 &\leq &\|\mathcal{D}_{n}\mathcal{C}_{n}^{-1}-D_{X}C_{X}^{-1}\|_{\mathcal{L}(H)}
 \to^{a.s.} 0,  \ n\to \infty .
\label{eqemsvtotsv}
\end{eqnarray}
Let us  define the following quantity:
\begin{equation}\Lambda_{k_{n}}^{\rho}=\sup_{1\leq j\leq k_{n}}(|\rho_{j}|^{2}-|\rho_{j+1}|^{2})^{-1},\label{ueev2}
\end{equation}
\noindent where $k_n$ denotes the truncation parameter introduced in Section \ref{sec:2}.
We now apply the methodology  of the proof of Lemma 4.3, on p. 104, and Corollary 4.3, on p. 107,  in \cite{Bosq00}, to obtain 
the strong-consistency of the empirical right and left eigenvectors, $\{\psi_{n,j},\ j\geq 1\}$ and $\{\widetilde{\psi}_{n,j},\ j\geq 1\}$ of $\rho,$   under the following additional assumption:

\medskip

\noindent \textbf{Assumption A4}. Consider $\left[\sup_{j\geq 1}|\rho_{j}|+ 
\sup_{j\geq 1}|\widehat{\rho}_{n,j}|\right]\leq 1.$

\medskip

\begin{lemma}
\label{lemeerlc}
Under \textbf{Assumptions A3--A4}, and   the conditions of \linebreak Theorem \ref{proposition5}\textbf{(ii)}, if
$\Lambda_{k_{n}}^{\rho}$   in  (\ref{ueev2}) is such that,  as $n\to \infty,$ $\Lambda_{k_{n}}^{\rho}=o\left(\frac{1}{M_{n}}\right),$ with 
\linebreak $\|\mathcal{D}_{n}\mathcal{C}_{n}^{-1}-D_{X}C_{X}^{-1}\|_{\mathcal{L}(H)}=\mathcal{O}\left( M_{n}\right),$ a.s.,  then, \begin{eqnarray}&&\sup_{1\leq j\leq k_{n}}\|\psi_{n,j}-\psi_{n,j}^{\prime }\|_{H}\to_{a.s.} 0, \quad \sup_{1\leq j\leq k_{n}}\|\widetilde{\psi}_{n,j}-
\widetilde{\psi}_{n,j}^{\prime }
\|_{H}\to_{a.s.} 0, \label{cre}
\end{eqnarray}
\noindent  where, for $j \geq 1,~ n\geq 2,$
$\psi_{n,j}^{\prime }= \mbox{sgn}\left\langle \psi_{n,j} , \psi_{j} \right\rangle_{H}\psi_{j}\quad \widetilde{\psi}_{n,j}^{\prime }= \mbox{sgn}\left\langle \widetilde{\psi}_{n,j}, \widetilde{\psi}_{j} 
\right\rangle_{H}\widetilde{\psi}_{j},$ with $\mbox{sgn}\langle \psi_{n,j}, \psi_j \rangle_H= \mathbf{1}_{\langle \psi_{n,j}, \psi_j \rangle_H\geq 0}-\mathbf{1}_{\langle \psi_{n,j}, \psi_j \rangle_H< 0}$ and $\mbox{sgn} \langle \widetilde{\psi}_{n,j}, \widetilde{\psi}_j \rangle_H= \mathbf{1}_{\langle \widetilde{\psi}_{n,j}, \widetilde{\psi}_j \rangle_H\geq 0}-\mathbf{1}_{\langle \widetilde{\psi}_{n,j}, \widetilde{\psi}_j \rangle_H< 0}$.
\end{lemma}
The proof of this lemma is given in Section \ref{secproof}. 

The following diagonal componentwise estimator 
$\widehat{\rho}_{k_{n}}$ of $\rho $ is  formulated:
\begin{equation}
\widehat{\rho}_{k_{n}}(x)=\sum_{j=1}^{k_{n}}
\widehat{\rho}_{n,j}\left\langle x,\psi_{n,j}
\right\rangle_{H}\widetilde{\psi}_{n,j},\quad \forall  x\in H.
\label{eqrhoest}
\end{equation}
\noindent The next result derives the strong-consistency of $\widehat{\rho}_{k_{n}}.$
\begin{theorem}
\label{sdsvde}
Under the conditions of Lemma \ref{lemeerlc}, if, as  $n\to \infty,$ \linebreak $k_{n}\Lambda_{k_{n}}^{\rho}=o\left(\frac{1}{M_{n}}\right),$ with
$\|\mathcal{D}_{n}\mathcal{C}_{n}^{-1}-D_{X}C_{X}^{-1}\|_{\mathcal{L}(H)}=\mathcal{O}\left( M_{n}\right),$ a.s.,  then, $
\|\widehat{\rho}_{k_{n}}-\rho\|_{\mathcal{L}(H)}\to_{a.s.}0,$ $n\to \infty.$
\end{theorem}

The proof of this result  is given  in Section \ref{secproof}. 
\vspace{-0.20cm}

\section{Proofs of the  results}
\label{secproof}
\subsection*{Proof of Lemma 1}

 Let us denote $\phi_{n,j}^{\prime} =sgn\langle \phi_{j}, \phi_{n,j} \rangle_{H} \phi_{j},$ where $sgn \langle \phi_{j}, \phi_{n,j} \rangle_{H} =  \boldsymbol{1}_{\langle \phi_{j}, \phi_{n,j} \rangle_H \geq 0} -   \boldsymbol{1}_{\langle \phi_{j}, \phi_{n,j} \rangle_H < 0}.$  Applying the triangle and  Cauchy--Schwarz inequalities, 
we obtain, as $n\to \infty,$
\begin{eqnarray}&&\sup_{x\in H,\ \|x\|_{H}\leq 1}\left\|\sum_{j=1}^{k_{n}}\left\langle \rho(x),\phi_{n,j}\right\rangle_{H}\phi_{n,j}-\rho(x)\right\|_{H}\nonumber\\&&\leq  \sup_{x\in H,\ \|x\|_{H}\leq 1}\left\|\sum_{j=1}^{k_{n}}\left\langle \rho(x),
\phi_{n,j}\right\rangle_{H}\phi_{n,j}-\left\langle \rho(x),
\phi_{n,j}^{\prime}\right\rangle_{H}\phi_{n,j}^{\prime}\right\|_{H}
\nonumber\\
&&
+\sup_{x\in H,\ \|x\|_{H}\leq 1}\left\|\sum_{j=k_{n}+1}^{\infty}\left\langle \rho(x),
\phi_{n,j}^{\prime}\right\rangle_{H}\phi_{n,j}^{\prime}\right\|_{H}\nonumber\end{eqnarray}\begin{eqnarray}
&&= \sup_{x\in H,\ \|x\|_{H}\leq 1}\left\|\sum_{j=1}^{k_{n}}\left\langle \rho(x),
\phi_{n,j}\right\rangle_{H}(\phi_{n,j}-\phi_{n,j}^{\prime})+\left\langle \rho(x),
\phi_{n,j}-\phi_{n,j}^{\prime}\right\rangle_{H}\phi_{n,j}^{\prime}\right\|_{H}\nonumber\\
&&+\sup_{x\in H,\ \|x\|_{H}\leq 1}\left\|\sum_{j=k_{n}+1}^{\infty}\left\langle \rho(x),
\phi_{n,j}^{\prime}\right\rangle_{H}\phi_{n,j}^{\prime}\right\|_{H}
\nonumber\\
&&
\leq\sup_{x\in H,\ \|x\|_{H}\leq 1}\sum_{j=1}^{k_{n}}\left|\left\langle \rho(x),
\phi_{n,j}\right\rangle_{H}\right|\|\phi_{n,j}-\phi_{n,j}^{\prime}\|_{H}\nonumber\\&&+\sup_{x\in H,\ \|x\|_{H}\leq 1}\left|\left\langle \rho(x),
\phi_{n,j}-\phi_{n,j}^{\prime}\right\rangle_{H}\right|\left\|\phi_{n,j}^{\prime}\right\|_{H}
\nonumber\\
&&
+\sup_{x\in H,\ \|x\|_{H}\leq 1}\left\|\sum_{j=k_{n}+1}^{\infty}\left\langle \rho(x),
\phi_{n,j}^{\prime}\right\rangle_{H}\phi_{n,j}^{\prime}\right\|_{H}\nonumber\\&&\leq\sum_{j=1}^{k_{n}}\left\|\rho
\right\|_{\mathcal{L}(H)}
\|\phi_{n,j}-\phi_{n,j}^{\prime}\|_{H}+\left\|\rho
\right\|_{\mathcal{L}(H)}
\|\phi_{n,j}-\phi_{n,j}^{\prime}\|_{H}\nonumber\\&&+\sup_{x\in H,\ \|x\|_{H}\leq 1}\left\|\sum_{j=k_{n}+1}^{\infty}\left\langle \rho(x),
\phi_{n,j}^{\prime}\right\rangle_{H}\phi_{n,j}^{\prime}\right\|_{H}=2\sum_{j=1}^{k_{n}}\left\|\rho
\right\|_{\mathcal{L}(H)}
\|\phi_{n,j}-\phi_{n,j}^{\prime}\|_{H}\nonumber\\&& +\sup_{x\in H,\ \|x\|_{H}\leq 1}\left\|\sum_{j=k_{n}+1}^{\infty}\left\langle \rho(x),
\phi_{n,j}^{\prime}\right\rangle_{H}\phi_{n,j}^{\prime}\right\|_{H}\nonumber\\
&&\leq 4\sqrt{2}\left\|\rho
\right\|_{\mathcal{L}(H)}k_{n}\Lambda_{k_{n}}\|\mathcal{C}_{n}-C_{X}\|_{\mathcal{S}(H)}\nonumber\\
&&
+\sup_{x\in H,\ \|x\|_{H}\leq 1}\left\|\sum_{j=k_{n}+1}^{\infty}\left\langle \rho(x),
\phi_{n,j}^{\prime}\right\rangle_{H}\phi_{n,j}^{\prime}\right\|_{H}, \label{eql1proof}
\end{eqnarray}
\noindent since, from  Corollary 4.3 in p.107 in  \cite{Bosq00}, \begin{equation}\sup_{1\leq j\leq k_{n}}\|\phi_{n,j}-\phi_{n,j}^{\prime}\|_{H}\leq 2\sqrt{2}\Lambda_{k_{n}}\|\mathcal{C}_{n}-C_{X}\|_{\mathcal{S}(H)}.\label{f454B00}\end{equation}
\noindent From  (\ref{eql1proof}), under the condition $$k_{n}\Lambda_{k_{n}}=o\left(\sqrt{\frac{n}{\ln(n)}}\right),\quad n\to \infty,$$
\noindent applying Theorem 1, we obtain 
$$\sup_{x\in H,\ \|x\|_{H}\leq 1}\left\|\rho(x)-\sum_{j=1}^{k_{n}}\left\langle \rho(x),\phi_{n,j}\right\rangle_{H}\phi_{n,j}\right\|_{H}\to_{a.s.} 0,\quad n\to \infty.$$

\subsection*{Proof of Theorem  2}

\noindent \textbf{(i)}\  Applying H\"older  and triangle inequalities, since $\rho=D_{X}C_{X}^{-1 }$ is bounded, from Theorem 1, under  $\sqrt{k_{n}}\Lambda_{k_{n}}=o\left(\frac{n^{1/4}}{(\ln(n))^{\beta }} \right)$ as $n\to \infty,$ for $\beta >1/2,$  
\begin{eqnarray}
&&\|\widetilde{\pi}^{k_{n}}\mathcal{D}_{n}\mathcal{C}_{n}^{-1}[\widetilde{\pi}^{k_{n}}]^{\star}-\widetilde{\pi}^{k_{n}}D_{X}C_{X}^{-1}[\widetilde{\pi}^{k_{n}}]^{\star}\|_{1}\nonumber\\
&&\leq \sqrt{k_{n}}\|\widetilde{\pi}^{k_{n}}\mathcal{D}_{n}\mathcal{C}_{n}^{-1}[\widetilde{\pi}^{k_{n}}]^{\star}-\widetilde{\pi}_{k_{n}}D_{X}C_{X}^{-1}[\widetilde{\pi}_{k_{n}}]^{\star}\|_{\mathcal{S}(H)}
\nonumber\\
&&\leq \sqrt{k_{n}}\|\widetilde{\pi}^{k_{n}}\mathcal{D}_{n}\mathcal{C}_{n}^{-1}[\widetilde{\pi}^{k_{n}}]^{\star}-\widetilde{\pi}^{k_{n}}D_{X}\mathcal{C}_{n}^{-1}[\widetilde{\pi}^{k_{n}}]^{\star}\|_{\mathcal{S}(H)}\nonumber\\
&&+
\sqrt{k_{n}}\|\widetilde{\pi}^{k_{n}}D_{X}\mathcal{C}_{n}^{-1}[\widetilde{\pi}^{k_{n}}]^{\star}
-\widetilde{\pi}_{k_{n}}D_{X}C_{X}^{-1}[\widetilde{\pi}_{k_{n}}]^{\star}\|_{\mathcal{S}(H)}\nonumber\\
&&=
\sqrt{k_{n}}\|\widetilde{\pi}^{k_{n}}(\mathcal{D}_{n}-D_{X})\mathcal{C}_{n}^{-1}[\widetilde{\pi}^{k_{n}}]^{\star}\|_{\mathcal{S}(H)}\nonumber\\
&&+
\sqrt{k_{n}}\|\widetilde{\pi}^{k_{n}}D_{X}C_{X}^{-1}\left[C_{X}\mathcal{C}_{n}^{-1}\mathcal{C}_{n}
-C_{X}C_{X}^{-1}\mathcal{C}_{n}\right]\mathcal{C}_{n}^{-1}[\widetilde{\pi}_{k_{n}}]^{\star}\|_{\mathcal{S}(H)}\nonumber\\
&&\leq \sqrt{k_{n}} C_{k_{n}}^{-1}\left[\|D_{X}-\mathcal{D}_{n}
\|_{\mathcal{S}(H)}+\|D_{X}C_{X}^{-1}\|_{\mathcal{L}(H)}
\|C_{X}-\mathcal{C}_{n}\|_{\mathcal{S}(H)}\right]\nonumber\\
&&\leq \sqrt{k_{n}}\Lambda_{k_{n}}\left[\|D_{X}-\mathcal{D}_{n}
\|_{\mathcal{S}(H)}+\|D_{X}C_{X}^{-1}\|_{\mathcal{L}(H)}
\|C_{X}-\mathcal{C}_{n}\|_{\mathcal{S}(H)}\right]
\label{iq}\\
&&\leq K\sqrt{k_{n}}\Lambda_{k_{n}}\left[
\|C_{X}-\mathcal{C}_{n}\|_{\mathcal{S}(H)}+\|D_{X}-\mathcal{D}_{n}
\|_{\mathcal{S}(H)}\right]\to_{a.s.} 0,\quad n\to \infty,\nonumber
\end{eqnarray}  \noindent for $\|\rho\|_{\mathcal{L}(H)}\leq K,$ $K\geq 1.$  Then, $$\|\widetilde{\rho}_{k_n}-\widetilde{\pi}^{k_{n}}\rho [\widetilde{\pi}^{k_{n}}]^{\star}\|_{1}\to^{a.s.} 0,\quad n\rightarrow \infty .$$

\noindent \textbf{(ii)}   Under \textbf{Assumptions A1--A2},
from Theorem 1,
\begin{eqnarray}
& &\left\| \mathcal{C}_{n} - C_{X} \right\|_{\mathcal{S}(H)} = \mathcal{O} \left(\left(\frac{\ln(n) }{n} \right)^{1/2} \right)~a.s.,
\nonumber\\
& &\left\| \mathcal{D}_n - D_{X} \right\|_{\mathcal{S}(H)} = \mathcal{O} \left(\left(\frac{\ln(n) }{n} \right)^{1/2} \right)~a.s.\nonumber
\end{eqnarray}
  Hence, from  equation (\ref{iq}),  as $n\to \infty,$
\begin{eqnarray}
&&\|\widetilde{\pi}^{k_{n}}\mathcal{D}_{n}\mathcal{C}_{n}^{-1}[\widetilde{\pi}^{k_{n}}]^{\star}-\widetilde{\pi}^{k_{n}}D_{X}C_{X}^{-1}[\widetilde{\pi}^{k_{n}}]^{\star}\|_{1}\to_{a.s.} 0.
\label{eqrepetida}
\end{eqnarray}
Let us now consider
 \begin{eqnarray}&&\|\widetilde{\rho}_{k_n}-\rho\|_{1}\leq \|\widetilde{\rho}_{k_n}-\widetilde{\pi}^{k_{n}}\rho [\widetilde{\pi}^{k_{n}}]^{\star}\|_{1}+\|\widetilde{\pi}^{k_{n}}\rho [\widetilde{\pi}^{k_{n}}]^{\star}-\rho\|_{1}.\label{eqtrishb}
\end{eqnarray}
\noindent From equation (\ref{eqrepetida}), the first term
at the right-hand side of  inequality (\ref{eqtrishb}) converges a.s. to zero. From
Lemma 1, $\widetilde{\pi}^{k_{n}}\rho [\widetilde{\pi}^{k_{n}}]^{\star}$ converges a.s. to $\rho,$ in $\mathcal{L}(H),$  as $n\to \infty.$ Since $\rho $ is trace operator, Dominated Covergence Theorem leads to $\|\widetilde{\pi}^{k_{n}}\rho [\widetilde{\pi}^{k_{n}}]^{\star}-\rho\|_{1}\to_{a.s.} 0,$  $n\to \infty,$ and  $$\|\widetilde{\rho}_{k_n}-\rho \|_{1}\to^{a.s.} 0,\quad n\to \infty.$$

\subsection*{Strong-consistency of the plug-in predictor}
\begin{corollary} 
Under the conditions of Theorem 2\textbf{(ii)},
\label{proposition6}
\begin{equation}
\|\widetilde{\rho}_{k_n}(X_{n-1})-\rho(X_{n-1})\|_{H} \to_{a.s.} 0,\quad n\rightarrow \infty.\label{eqfr}
\end{equation}

\end{corollary}

\begin{proof}  Let $\|X_{0}\|_{\infty, H}= \inf\left\{
c;\ P(\|X_{0}\|_{H}>c)=0\right\}<\infty,$ under \textbf{Assumption A1}.
From Theorem 2\textbf{(ii)}, we then have    
\begin{eqnarray}&&
\|\widetilde{\rho}_{k_n}-\rho\|_{\mathcal{L}(H)}\to^{a.s.} 0,\quad
n\to \infty,\quad \mbox{and}\label{eqcor}\\
&&\|\widetilde{\rho}_{k_n}(X_{n-1})-\rho(X_{n-1})\|_{H}\leq \|\widetilde{\rho}_{k_n}-\rho\|_{\mathcal{L}(H)}\|X_{0}\|_{\infty, H}\to^{a.s.} 0, \ n\to \infty.\nonumber
\end{eqnarray} 

\end{proof}

\subsection*{Proof of Lemma 2}  
 Under \textbf{Assumption A3},  $\rho^{\star}\rho,$  
$[\mathcal{D}_{n}\mathcal{C}_{n}^{-1}]^{\star}[\mathcal{D}_{n}\mathcal{C}_{n}^{-1}],$
$\rho\rho^{\star}$ and \linebreak $[\mathcal{D}_{n}\mathcal{C}_{n}^{-1}][\mathcal{D}_{n}\mathcal{C}_{n}^{-1}]^{\star}$ are  self-adjoint compact operators, admitting the following diagonal spectral series representations in $H:$
\begin{eqnarray}
\rho^{\star}\rho &\underset{H}{=}&\sum_{j=1}^{\infty}|\rho_{j}|^{2}
\psi_{j}\otimes \psi_{j}\quad [\mathcal{D}_{n}\mathcal{C}_{n}^{-1}]^{\star}[\mathcal{D}_{n}\mathcal{C}_{n}^{-1}]\underset{H}{=}
\sum_{j=1}^{n}|\widehat{\rho}_{n,j}|^{2}\psi_{n,j}\otimes \psi_{n,j}\nonumber\\ \label{sdrt}\\
\rho\rho^{\star}&\underset{H}{=}&\sum_{j=1}^{\infty}|\rho_{j}|^{2}
\widetilde{\psi}_{j}\otimes\widetilde{\psi}_{j}\quad \mathcal{D}_{n}\mathcal{C}_{n}^{-1}[\mathcal{D}_{n}\mathcal{C}_{n}^{-1}]^{\star}
\underset{H}{=}\sum_{j=1}^{n}|\widehat{\rho}_{n,j}|^{2}\widetilde{\psi}_{n,j}\otimes\widetilde{\psi}_{n,j}.\nonumber\\\label{dsdrho2}
\end{eqnarray}
 From  (\ref{sdrt}),  applying   triangle inequality,
\begin{eqnarray}
&&\|\rho^{\star}\rho (\psi_{n,j})-|\rho_{j}|^{2}\psi_{n,j}\|_{H}\leq  \|\rho^{\star}\rho (\psi_{n,j})-
[\mathcal{D}_{n}\mathcal{C}_{n}^{-1}]^{\star}[\mathcal{D}_{n}\mathcal{C}_{n}^{-1}](\psi_{n,j})\|_{H}
\nonumber\\
&&+\|[\mathcal{D}_{n}\mathcal{C}_{n}^{-1}]^{\star}[\mathcal{D}_{n}\mathcal{C}_{n}^{-1}](\psi_{n,j})-|\rho_{j}|^{2}\psi_{n,j}\|_{H}\nonumber\\
&&\leq 2\|\rho^{\star}\rho-[\mathcal{D}_{n}\mathcal{C}_{n}^{-1}]^{\star}[\mathcal{D}_{n}\mathcal{C}_{n}^{-1}]\|_{\mathcal{L}(H)}.
\label{eqdif}
\end{eqnarray}

On the other hand, \begin{eqnarray}
&&\|\psi_{n,j}-\psi_{n,j}^{\prime }\|_{H}^{2}=
\sum_{l=1}^{\infty}\left[\left\langle \psi_{n,j},\psi_{l} \right\rangle_{H}-\mbox{sgn}\left\langle \psi_{n,j} , \psi_{l} \right\rangle_{H}\left\langle \psi_{j}, \psi_{l}\right\rangle_{H}\right]^{2}
\nonumber\\
&&=\sum_{l\neq j}\left[\left\langle \psi_{n,j},\psi_{l} \right\rangle_{H}\right]^{2}+\left[\left\langle \psi_{n,j},\psi_{j} \right\rangle_{H}-\mbox{sgn}\left\langle \psi_{n,j} , \psi_{j} \right\rangle_{H}\right]^{2}\nonumber\\
&&=\sum_{l\neq j}\left[\left\langle \psi_{n,j},\psi_{l} \right\rangle_{H}\right]^{2}+\left[1-\left|\left\langle \psi_{n,j},\psi_{j} \right\rangle_{H}\right|\right]^{2}\nonumber\\
&&=\sum_{l\neq j}\left[\left\langle \psi_{n,j},\psi_{l} \right\rangle_{H}\right]^{2}+\sum_{l=1}^{\infty}\left[\left\langle \psi_{n,j},\psi_{l} \right\rangle_{H}\right]^{2}-2\left|\left\langle \psi_{n,j},\psi_{j} \right\rangle_{H}\right|+\left|\left\langle \psi_{n,j},\psi_{j} \right\rangle_{H}\right|^{2}\nonumber\\
&&\leq 2\sum_{l\neq j}\left[\left\langle \psi_{n,j},\psi_{l} \right\rangle_{H}\right]^{2}.
\label{eqdif2}
\end{eqnarray}
Furthermore,
\begin{eqnarray}
&&\|\rho^{\star}\rho (\psi_{n,j})-|\rho_{j}|^{2}\psi_{n,j}\|_{H}^{2}=
\sum_{l=1}^{\infty}\left[\left\langle \psi_{n,j},|\rho_{l}|^{2}\psi_{l}\right\rangle_{H}-\left\langle \psi_{n,j},|\rho_{j}|^{2}\psi_{l}\right\rangle_{H}\right]^{2}\nonumber\\
&&\geq \min_{l\neq j}\left||\rho_{l}|^{2}-|\rho_{j}|^{2}\right|^{2}\sum_{l\neq j}\left[\left\langle \psi_{n,j},\psi_{l} \right\rangle_{H}\right]^{2}\nonumber\\
&&\geq \min_{l\neq j}\left||\rho_{l}|^{2}-|\rho_{j}|^{2}\right|^{2}\frac{1}{2}\|\psi_{n,j}-\psi_{n,j}^{\prime }\|_{H}^{2}\nonumber\\
&&\geq \alpha_{j}^{2}\frac{1}{2}\|\psi_{n,j}-\psi_{n,j}^{\prime }\|_{H}^{2},\label{eqdif3}
\end{eqnarray}
\noindent where $\alpha_{1}=(|\rho_{1}|^{2}-|\rho_{2}|^{2}),$ and 
\begin{eqnarray}
&&\alpha_{j}=\min\left(|\rho_{j-1}|^{2}-|\rho_{j}|^{2},|\rho_{j}|^{2}-|\rho_{j+1}|^{2}\right),\quad j\geq 2.\label{alphaj}
\end{eqnarray}

From equations (\ref{eqdif}) and  (\ref{eqdif3}), we have
\begin{eqnarray}
&&\|\psi_{n,j}-\psi_{n,j}^{\prime }\|_{H}\leq 
a_{j}\|\rho^{\star}\rho-
[\mathcal{D}_{n}\mathcal{C}_{n}^{-1}]^{\star}[\mathcal{D}_{n}\mathcal{C}_{n}^{-1}]\|_{\mathcal{L}(H)},\label{eqdif4}
\end{eqnarray}
\noindent where  $a_{1}=2\sqrt{2}(|\rho_{1}|^{2}-|\rho_{2}|^{2})^{-1},$ and 
\begin{eqnarray}
&& a_{j}=2\sqrt{2}\max\left[\left(|\rho_{j-1}|^{2}-|\rho_{j}|^{2}\right)^{-1},\left(|\rho_{j}|^{2}-|\rho_{j+1}|^{2}\right)^{-1}\right].
\label{alphajb}
\end{eqnarray}
  
  In a similar way, considering the operators $\rho\rho^{\star}$ and 
  $\widehat{\rho}_{k_{n}}\widehat{\rho}^{\star}_{k_{n}}$ instead of 
  $\rho^{\star}\rho$ and $\widehat{\rho}^{\star}_{k_{n}}\widehat{\rho}_{k_{n}},$ respectively, we can obtain
  \begin{eqnarray}
&&\|\widetilde{\psi}_{n,j}-\widetilde{\psi}_{n,j}^{\prime }\|_{H}\leq 
a_{j}\|\rho\rho^{\star}-[\mathcal{D}_{n}\mathcal{C}_{n}^{-1}][\mathcal{D}_{n}
\mathcal{C}_{n}^{-1}]^{\star}
\|_{\mathcal{L}(H)}.\label{eqdif4b}
\end{eqnarray}
  
From equations (\ref{eqdif4})--(\ref{eqdif4b}), 
\begin{eqnarray}&&\sup_{1\leq j\leq k_{n}}\|\psi_{n,j}-\psi_{n,j}^{\prime }\|_{H}\leq 2\sqrt{2}\Lambda_{k_{n}}^{\rho}
\|\rho^{\star}\rho-[\mathcal{D}_{n}\mathcal{C}_{n}^{-1}]^{\star}[\mathcal{D}_{n}\mathcal{C}_{n}^{-1}]\|_{\mathcal{L}(H)}\nonumber\\
&&\sup_{1\leq j\leq k_{n}}\|\widetilde{\psi}_{n,j}-\widetilde{\psi}_{n,j}^{\prime }
\|_{H}\leq 2\sqrt{2}\Lambda_{k_{n}}^{\rho}\|\rho\rho^{\star}-[\mathcal{D}_{n}\mathcal{C}_{n}^{-1}][\mathcal{D}_{n}\mathcal{C}_{n}^{-1}]^{\star}\|_{\mathcal{L}(H)}.\nonumber\\\label{efveeiv2}
\end{eqnarray}

 Since, under \textbf{Assumption A4},
\begin{eqnarray}&&\|\rho^{\star}\rho-[\mathcal{D}_{n}\mathcal{C}_{n}^{-1}]^{\star}[\mathcal{D}_{n}\mathcal{C}_{n}^{-1}]\|_{\mathcal{L}(H)}\leq \|\mathcal{D}_{n}\mathcal{C}_{n}^{-1}-D_{X}C_{X}^{-1}\|_{\mathcal{L}(H)}\nonumber\\
&&\|\rho\rho^{\star}-[\mathcal{D}_{n}\mathcal{C}_{n}^{-1}][\mathcal{D}_{n}
\mathcal{C}_{n}^{-1}]^{\star}\|_{\mathcal{L}(H)}\leq \|\mathcal{D}_{n}\mathcal{C}_{n}^{-1}-D_{X}C_{X}^{-1}\|_{\mathcal{L}(H)},
\label{efveeiv2hh}
\end{eqnarray}
\noindent we obtain from  Proposition 1, and  (\ref{efveeiv2})--(\ref{efveeiv2hh}), keeping in mind that, \linebreak as $n\to \infty,$ 
 $\Lambda_{k_{n}}^{\rho}=o\left(\frac{1}{M_{n}}\right),$ with $\|\mathcal{D}_{n}\mathcal{C}_{n}^{-1}-D_{X}C_{X}^{-1}\|_{\mathcal{L}(H)}=\mathcal{O}(M_{n})$ a.s.,
\begin{eqnarray}&&\sup_{1\leq j\leq k_{n}}\|\psi_{n,j}-\psi_{n,j}^{\prime }\|_{H}\to_{a.s.} 0,\quad n\to \infty \nonumber\\
&&\sup_{1\leq j\leq k_{n}}\|\widetilde{\psi}_{n,j}-\widetilde{\psi}_{n,j}^{\prime }
\|_{H}\to_{a.s.} 0,\quad n\to \infty. \nonumber
\end{eqnarray}

\subsection*{Proof of Theorem 3}

Let us consider
\begin{eqnarray}
&&\sup_{x\in H;\ \|x\|_{H}\leq 1}\|\widehat{\rho}_{k_{n}}(x)-\rho(x)\|_{H}\leq \sup_{x\in H;\ \|x\|_{H}\leq 1}\left(\|\widehat{\rho}_{k_{n}}
\widetilde{\Pi}^{k_{n}}(x)-\rho\Pi^{k_{n}}(x)\|_{H}\right.
\nonumber\\
&&\left.+
\|\rho\Pi^{k_{n}}(x)-\rho\widetilde{\Pi}^{k_{n}}(x)\|_{H}
+\|\rho\widetilde{\Pi}^{k_{n}}(x)-\rho(x)\|_{H}\right)\nonumber\\
&&= \sup_{x\in H;\ \|x\|_{H}\leq 1} \left( a_{n}(x)+b_{n}(x)+c_{n}(x)\right)\nonumber\\
&&\leq \sup_{x\in H;\ \|x\|_{H}\leq  1}a_{n}(x)+\sup_{x\in H;\ \|x\|_{H}\leq 1}b_{n}(x)+\sup_{x\in H;\ \|x\|_{H}\leq 1}c_{n}(x),
\label{eq1pmt}
\end{eqnarray}
\noindent where $\widetilde{\Pi}^{k_{n}}$ denotes the projection operator into the subspace of $H$ generated by $\{\psi_{n,j},\ j\geq 1\},$ and $\Pi^{k_{n}}$ is the projection operator into the subspace of $H$ generated by $\{\psi_{j},\ j\geq 1\}.$

As given in Section 4 of the paper, under \textbf{Assumption A3}, 
from Lemma 4.2 on p. 103 in \cite{Bosq00} (see also Proposition 1, and equation (17) of the paper),
\begin{eqnarray}
 \sup_{j\geq 1}\left|\widehat{\rho}_{n,j}-\rho_{j}\right|
 &\leq &\|\mathcal{D}_{n}\mathcal{C}_{n}^{-1}-D_{X}C_{X}^{-1}\|_{\mathcal{L}(H)}
 \to^{a.s.} 0,  \ n\to \infty .
\label{eqemsvtotsv}
\end{eqnarray}
 
 Applying now the triangle and Cauchy--Schwarz inequalities, from equations (\ref{efveeiv2})  and (\ref{eqemsvtotsv}),  as $n\to \infty,$ 
\begin{eqnarray}&&
\sup_{x\in H;\ \|x\|_{H}\leq 1}
a_{n}(x)=\sup_{x\in H;\ \|x\|_{H}\leq 1}\|\widehat{\rho}_{k_{n}}
\widetilde{\Pi}^{k_{n}}(x)-\rho\Pi^{k_{n}}(x)\|_{H}\nonumber\\
&&\leq \sup_{x\in H;\ \|x\|_{H}\leq 1}\sum_{j=1}^{k_{n}}|\widehat{\rho}_{n,j}-\rho_{j}|\left|\left\langle x,\psi_{n,j}\right\rangle_{H}\right|\|\widetilde{\psi}_{n,j}\|_{H}
\nonumber\\
&&+\sup_{x\in H;\ \|x\|_{H}\leq 1}|\rho_{j}|\left|\left\langle x,\psi_{n,j}-\psi_{n,j}^{\prime}\right\rangle_{H}\right|\|\widetilde{\psi}_{n,j}\|_{H}
\nonumber\\
&&
+\sup_{x\in H;\ \|x\|_{H}\leq 1}|\rho_{j}|\left|\left\langle x,\psi_{n,j}^{\prime}\right\rangle_{H}\right|\|\widetilde{\psi}_{n,j}-\widetilde{\psi}_{n,j}^{\prime}\|_{H}
\nonumber\\
&&\leq \sum_{j=1}^{k_{n}}|\widehat{\rho}_{n,j}-\rho_{j}|+|\rho_{j}|\left[\|\psi_{n,j}-\psi_{n,j}^{\prime}\|_{H}+\|\widetilde{\psi}_{n,j}-\widetilde{\psi}_{n,j}^{\prime}\|_{H}
\right]\nonumber\\
&&\leq k_{n}\|\mathcal{D}_{n}\mathcal{C}_{n}^{-1}-D_{X}C_{X}^{-1}\|_{\mathcal{L}(H)}\nonumber\\
&&+k_{n}\|\rho\|_{\mathcal{L}(H)}\left[2\sqrt{2}\Lambda_{k_{n}}^{\rho}
\|\rho^{\star}\rho-[\mathcal{D}_{n}\mathcal{C}_{n}^{-1}]^{\star}[\mathcal{D}_{n}\mathcal{C}_{n}^{-1}]\|_{\mathcal{L}(H)}\right]
\nonumber\\
&&+k_{n}\|\rho\|_{\mathcal{L}(H)}\left[
2\sqrt{2}\Lambda_{k_{n}}^{\rho}\|\rho\rho^{\star}-[\mathcal{D}_{n}
\mathcal{C}_{n}^{-1}][\mathcal{D}_{n}\mathcal{C}_{n}^{-1}]^{\star}\|_{\mathcal{L}(H)}\right],\label{eqcaxas}
\end{eqnarray}
\noindent which converges a.s. to zero, under \textbf{Assumption A4} (see also equation (\ref{efveeiv2hh})),  since \begin{equation}k_{n}\Lambda_{k_{n}}^{\rho}=o\left(\frac{1}{M_{n}}\right),\quad n\to \infty,
\label{eqasconvzero}
\end{equation}
\noindent  
 with $\|\mathcal{D}_{n}\mathcal{C}_{n}^{-1}-D_{X}C_{X}^{-1}\|_{\mathcal{L}(H)}=\mathcal{O}(M_{n})$ a.s., as $n\to \infty .$
 Applying  triangle and Cauchy--Schwarz inequalities, from equation (\ref{efveeiv2}), in a similar way to (\ref{eqcaxas}), under  \textbf{Assumption A4} and (\ref{eqasconvzero}), we obtain
\begin{eqnarray}&&\sup_{x\in H;\ \|x\|_{H}\leq 1}
b_{n}(x)=\sup_{x\in H;\ \|x\|_{H}\leq 1}\|\rho\Pi^{k_{n}}(x)-
\rho\widetilde{\Pi}^{k_{n}}(x)\|_{H}
\nonumber\\
&&\leq \sup_{x\in H;\ \|x\|_{H}\leq 1}\sum_{j=1}^{k_{n}}
\|x\|_{H}\|\psi_{n,j}^{\prime}-\psi_{n,j}\|_{H}|\rho_{j}|\|\widetilde{\psi}^{\prime}_{n,j}\|_{H}\nonumber\\
&&+\sup_{x\in H;\ \|x\|_{H}\leq 1}\|x\|_{H} \|\|\psi_{n,j} \|_{H}\left\|\rho\left(\widetilde{\psi}^{\prime}_{n,j}-\widetilde{\psi}_{n,j}\right) \right\|_{H} \nonumber\\
&& \leq \sum_{j=1}^{k_{n}}\|\psi_{n,j}^{\prime}-\psi_{n,j}\|_{H}|\rho_{j}|+\|\rho\|_{\mathcal{L}(H)}\|\widetilde{\psi}^{\prime}_{n,j}-\widetilde{\psi}_{n,j}\|_{H}\nonumber\\
&& \leq \|\rho\|_{\mathcal{L}(H)}\sum_{j=1}^{k_{n}}\|\psi_{n,j}^{\prime}-\psi_{n,j}\|_{H}+\|\widetilde{\psi}^{\prime}_{n,j}-\widetilde{\psi}_{n,j}\|_{H}\to_{a.s.} 0,\ n\to \infty.
\label{eqcaxas2}
\end{eqnarray}

In a similar way to the proof of  Lemma 1, from  (\ref{efveeiv2}), under  \textbf{Assumption A4} and (\ref{eqasconvzero}), as $n\to \infty,$
\begin{eqnarray}
&&\sup_{x\in H;\ \|x\|_{H}\leq 1}c_{n}(x)=\sup_{x\in H;\ \|x\|_{H}\leq 1}\|\rho\widetilde{\Pi}^{k_{n}}(x)-\rho(x)\|_{H}\label{eqcaxas2b}\\
&&\leq 
2\left\|\rho
\right\|_{\mathcal{L}(H)}\sum_{j=1}^{k_{n}}
\|\psi_{n,j}-\psi_{n,j}^{\prime}\|_{H}
\nonumber\\
&&+\sup_{x\in H;\ \|x\|_{H}\leq 1}\left\|\sum_{j=k_{n}+1}^{\infty}\left\langle \rho(x),
\psi_{n,j}^{\prime}\right\rangle_{H}\psi_{n,j}^{\prime}\right\|_{H}
\to_{a.s} 0.\nonumber
\end{eqnarray}
 \noindent 
From equations (\ref{eqcaxas})--(\ref{eqcaxas2b}),   we obtain the desired result.
\vspace{-0.20cm}
\section*{Acknowledgments}
This work has been supported in part by project MTM2015--71839--P (co-funded by Feder funds),
of the DGI, MINECO, Spain.

\end{document}